\documentclass{amsart}
\usepackage{amssymb}
\usepackage{color}
\usepackage{graphicx}
\usepackage{amsmath,amscd}
\usepackage{psfrag}
\usepackage[all,cmtip]{xy}

\newcommand{\D}{\mathbb{D}}

\renewcommand{\S}{\mathbb{S}}

\newtheorem{proposition}{Proposition}
\newtheorem{theorem}[proposition]{Theorem}
\newtheorem{definition}[proposition]{Definition}
\newtheorem{lemma}[proposition]{Lemma}

\newtheorem{remark}[proposition]{Remark}

\newtheorem{example}[proposition]{Examples}

\begin{document}

\title{Iso-contact embeddings of manifolds in co-dimension $2$}
\date{\today}

\thanks{}

\author{Dishant M. Pancholi}
\address{Institute of Mathematical Sciences,
IV, Cross Road, CIT Campus, 
Taramani, 
Chennai 600133,
Tamilnadu, India.}
\email{dishant@imsc.res.in}

\author{Suhas Pandit}
\address{Indian Institute of Technology Madras,
IIT PO.Chennai, 600036, Tamilnadu, India.}
\email{suhas@iitm.ac.in}




\begin{abstract}
The purpose of this article is to study co-dimension $2$ iso-contact embeddings of closed contact
manifolds. We first show that a closed contact manifold $(M^{2n-1}, \xi_M)$ iso-contact embeds in 
a contact manifold $(N^{2n+1}, \xi_N),$ provided $M$ contact embeds in $(N, \xi_N)$ with a trivial
normal bundle and the
contact structure induced on $M$ via this embedding is homotopic as an almost-contact structure
to $\xi_M.$ We apply this result to first establish that a closed contact $3$--manifold having no $2$--torsion
in its second integral cohomology iso-contact embeds in the standard contact $5$--sphere if
and only if the first Chern class of the contact structure is zero. Finally, we discuss
iso-contact embeddings of closed simply connected contact $5$--manifolds.

\end{abstract}

\maketitle
\section{Introduction}\label{sec:intro}

The study of embeddings of manifolds in Euclidean spaces has been a very classical and well studied topic 
which has lead to the development of many important tools in geometric topology. 
H.Whitney in \cite{Wh} established that every smooth $n$--manifold admits an embedding in $\mathbb{R}^{2n}.$  He also demonstrated that $\mathbb{R}P^2$ does not admit
an embedding in $\mathbb{R}^3$ there by establishing that this result  
is optimal in general. However,  M. Hirsch generalized the result  for odd dimensional closed orientable manifolds to
establish that every $(2n+1)$-- dimensional manifold admits an embedding in $\mathbb{R}^{4n -1}.$ This, in particular, implies that every closed orientable $3$--manifold admits an embedding in $\mathbb{R}^5.$ On the other hand, J. Nash in \cite{Na} established that every closed Riemannian $n$--manifold admits a $c^1$--isometric embedding in $\frac{n}{2}(3n+11)$--dimensional flat Euclidean space. This initiated
the study of embeddings of manifolds preserving a given geometric structure.  

In this article, we study \emph{iso-contact} embeddings of contact manifolds.
This study was formally initiated by M. Gromov.  See, for example,  \cite{Gr}
for Gromov's approach to the iso-contact embedding problem.

Recall that by a contact structure on a manifold $M$, we mean a nowhere integrable hyperplane field $\xi$ on $M$. 
The contact structure is said to be co-orientable,  provided $\xi$ is the kernel of a $1$--form defined on $M.$
A contact manifold $M$ with the contact structure $\xi$ is denoted by the pair $(M, \xi).$  
When $\xi$ is co-oriented and $\xi$ is the kernel of a $1$--form $\alpha$ defined on $M,$ then we 
also denote the contact manifold $(M, \xi)$ by the pair $(M, Ker \{\alpha \}).$ 
In this article, we will always work with co-orientable contact structures defined on  orientable manifolds. 

 Now, let $(M_1, Ker \{\alpha_1\})$ and $(M_2,  Ker \{\alpha_2 \})$ be two contact manifolds. We say that $(M_1, Ker \{\alpha_1\})$ admits an \emph{iso-contact embedding} in $(M_2, Ker \{\alpha_2 \}),$ 
provided there exists a smooth embedding $f:M_1 \hookrightarrow M_2$ such that $f^*\alpha_2 = g \alpha_1,$ for some 
everywhere positive function $g:M \rightarrow \mathbb{R}.$ In case, a manifold  $M_1$ admits an embedding into a contact manifold $(M_2, \xi_2)$ such that
the restriction of $\xi_2$ to $M_1$ is a contact structure on $M_1,$ we say $M_1$ admits  \emph{contact embedding} in 
$(M_2,\xi_2).$

  Gromov in \cite{Gr} using his convex integration technique -- which generalizes the technique developed by  Nash in \cite{Na} -- essentially provided a  complete understanding regarding iso-contact embeddings when the co-dimension of the embedding is bigger than or equal to $4.$ He established that the question of
contact embeddings in any positive co-dimension and the question of
iso-contact embeddings of closed contact manifold manifolds in co-dimension bigger
than equal to four  abides by the $h$--principle. This
essentially means that whenever there are no formal bundle theoretic  obstructions 
for finding a contact embedding, there is indeed, a contact embedding. 
See \cite{EM} for more on  $h$--principle. 

Our main  focus is  on understanding   iso-contact embeddings
of closed contact  manifolds in co-dimension $2.$ In this  co-dimension 
the techniques developed by Gromov in \cite{Gr} are generally not sufficient
for a complete answer. The systematic study of co-dimension $2$  iso-contact embeddings of closed contact manifolds  was initiated  by J. Etnyre and R. Fukuwara in \cite{EF}.\footnote{In \cite{EF} the term contact embedding means iso-contact embedding. We on the other hand follow the conventions from \cite{EM}.} Iso-contact embeddings of contact $3$--manifolds  in 
$M \times S^2$ were also constructed  in \cite{NP} using the fact that the co-tangent bundle of any closed orientable $3$--manifold is trivial.




For a contact manifold $(M_1, Ker \{ \alpha_1\})$ to admit an iso-contact embedding in the manifold 
$(M_2, Ker \{\alpha_2 \}),$ 
there must exist a smooth embedding $f$ of $M_1$ in $M_2$ and a monomorphism $F:TM_1 \rightarrow TM_2$  which covers
the map $f$ and satisfies the property that the bundle $(F_*\xi_1,F_*d \alpha_1)$ is a conformal symplectic  sub-bundle of 
$(M_2, Ker \{\alpha_2 \}).$  If such a pair $(f,F)$ exists, then we say that we have a formal iso-contact embedding of
$(M_1, Ker \{\alpha_1 \})$ in $(M_2, Ker \{\alpha_2 \}).$ We refer to \cite[Chpt-12]{EM} for more on formal iso-contact
embeddings.

Let $(N, \xi)$ be a contact manifold. Following \cite{EM}, we say that the problem of
iso-contact embeddings of a collection $\mathcal{A}$ of contact manifolds abide by the $h$-principle, provided every formal 
iso-contact embedding
of a contact manifold $(M,\xi_M) \in \mathcal{A}$ in $(N, \xi)$  
can be isotoped to an iso-contact embedding of $(M, \xi_M)$ in $(N, \xi)$. 

Questions related to iso-contact embeddings of closed contact manifolds in an
arbitrary contact manifold $(N, \xi)$ abide by the $h$-principle provided, either the co-dimension of 
the embedding is bigger than $4$ or the target manifold is overtwisted in the sense of \cite{BEM}. See for example,
\cite[Theorem~12.3.1]{EM} to understand the case of iso-contact embeddings when  the co-dimension  of the
embedding is bigger than or equal to $4.$  Iso-contact embeddings were first studied by Gromov. 
He established in  \cite{Gr} the $h$-principle for iso-contact embeddings for the category of open
contact manifolds provided the co-dimension of the embeddings is bigger than or equal to $2$. 
As a result, he also obtained  the  $h$-principle for closed contact manifolds, provided the
co-dimension of the embeddings is at least $4.$
Gromov in \cite{Gr} also established the $h$-principle for co-dimension $2$ immersions. 
To clearly understand  the case of iso-contact embeddings in overtwisted contact manifolds, see  discussions 
related to iso-contact embeddings  in \cite{EF} and \cite{EL}. In particular, in \cite{EL} it is shown
that every closed contact $3$--manifold admits an iso-contact embedding in an overtwisted contact $\S^2 \times \S^3.$
We would also like to point out that A. Mori in \cite{Mr} also produced iso-contact embeddings of all
contact $3$--manifolds in the contact manifold $\left(\mathbb{R}^7, Ker \{dz + \displaystyle\sum_{i=1}^{3} x_i d y_i \}\right)$
using open books and D. Martinez-Torres in \cite{Ma} 
produced an iso-contact embedding of any contact manifold $M^{2n+1}$ in 
$\left(\mathbb{R}^{4n+3}, Ker \{ dz + \displaystyle\sum_{i = 1}^{2n+1} x_i d y_i  \}\right).$

%
In this article, we first  establish an $h$--principle type result for co-dimension $2$  iso-contact embeddings of  
closed manifolds. In order to state this  result, we need the notion of \emph{overtwisted}
contact manifolds due to M. Borman, Y. Eliashberg and E. Murphy discussed in \cite{BEM}.

Recall that a contact manifold $(M, \xi)$ is said to be overtwisted, provided it admits an iso-contact embedding
of an overtwisted ball. For a precise definition of an overtwisted ball,  refer  \cite{BEM}. 
For the purpose of this article what is important is the following fact established in \cite{BEM}:

In every homotopy class of  almost contact structures, there exists a unique overtwisted contact structure up to isotopy.

Here, by an almost-contact structure on a manifold $M,$ we mean a hyperplane-field 
$\xi$ together with a conformal class of a symplectic structure on it.
We would like to remark that a contact structure $Ker\{\alpha\}$  can be  naturally
regarded as   an almost-contact structure. This is because $d \alpha$ restricted to $Ker \{\alpha\}$ provides the conformal class of symplectic structure on the hyperplane $Ker\{\alpha\}.$

We will always  denote the unique overtwisted contact structure in the homotopy class of  an almost contact structure $\xi$ on a manifold $M$ by  $\xi^{ot}.$ Now, we state our main result of this article.

\begin{theorem}\label{thm:h-principle}
 Let $(M, \xi_M)$ be a closed contact manifold. If $(M, \xi_M^{ot})$ admits an iso-contact embedding in a 
 contact manifold $(N, \xi_N)$ with the trivial normal bundle, then so does
  $(M, \xi_M).$ 
\end{theorem}

Our proof of the Theorem~\ref{thm:h-principle} relies on
certain  \emph{flexibility} discovered in 
iso-contact embeddings of contact manifolds in
neighborhoods of a special class of  closed contact overtwisted manifolds. which are assumed to be embedded in a given  contact manifold. 
See the Proposition~\ref{pro:h-principle_weak} for a precise statement. 

Next, we discuss some  applications of the Theorem~\ref{thm:h-principle}. 
We first discuss co-dimension $2$ iso-contact embeddings of contact manifolds
in the standard contact spheres. 
Recall that by the standard contact structure 
$\xi_{std}$ on the unit sphere $\S^{2n-1} \subset \mathbb{R}^{2n},$ 
we mean the kernel of the $1$-form 
$ \displaystyle\sum_{i =1}^{n} x_i d y_i - y_i d x_i$ restricted to $\S^{2n-1}.$

The techniques developed to establish the Theorem~\ref{thm:h-principle} also
establishes the following very useful proposition.

\begin{proposition}\label{cor:h-principle_for_embedding_in_spheres}
A closed contact manifold $(M^{2n-1}, \xi_M)$ admits an iso-contact embedding in the
standard contact sphere $(\S^{2n+1}, \xi_{std})$ if and only if $M$ admits
a contact embedding   in $(\S^{2n+1},\xi_{std})$ and the induced contact
structure on $M$ by the embedding  is
 homotopic to $\xi_M$ as an almost-contact structure. 
\end{proposition}

There are many interesting classes of smooth manifolds  which admit smooth co-dimension $2$ embeddings 
in the standard spheres. 
For example, as mentioned earlier, Hirsch in \cite{Hi} showed that every closed smooth $3$--manifold admits a smooth embedding in $\S^5.$  
There are now many proofs of this result. See for example, \cite{HLM} for
what is now known as braided embedding  and \cite{PPS} for embeddings using open books.

N. Kasuya in \cite{Ka} first observed that not all contact $3$--manifolds admit iso-contact embeddings in 
the standard contact $\S^5.$ He showed that the necessary condition for the existence of such an embedding
is that the first Chern class of the contact structure must be zero. In \cite{Ka}, Kasuya also
established that every closed contact $3$--manifold $(M, \xi)$ admits an iso-contact embedding in
some contact $\mathbb{R}^5.$

 In \cite{EF}, Etnyre and Fukuwara obtained many interesting iso-contact embedding
results. One of the most striking  result which they established  states that every overtwisted contact $3$--manifold $(M, \xi_{ot})$ with no $2$--torsion in second integral cohomology iso-contact embeds in 
the standard contact $\S^5$ if and only if the first Chern class of the overtwisted contact structure 
$\xi_{ot}$ is zero. 

Applying the Theorem~\ref{thm:h-principle} about  iso-contact embeddings in spheres 
and a result about iso-contact embeddings of overtwisted $3$--manifolds in $\S^5$ proved in  \cite{EF}, we 
establish the following:

\begin{theorem}\label{thm:cont_3-fold_in_S^5}
 Let $M$ be a closed orientable $3$--manifold. Then, we have the following:
 \begin{enumerate}
 \item 
 In case, $M$ has no $2$--torsion in $H^2(M, \mathbb{Z}),$ then $M$ together with 
 any contact structure $\xi$ on it admits an  iso-contact embedding  in $(\S^5, \xi_{std})$ 
 if and only if the first Chern class $c_1(\xi)$ is zero. 
 \item 
 In case, $M$ has a $2$--torsion in $H^2(M, \mathbb{Z}),$ then
 there exits a homotopy class $[\xi]$ of plane fields  on $M$ such that $M$ together with
 any contact structure homotopic to a plane field belonging to the class $[\xi]$ 
 over a $2$--skeleton of $M$ admits an
 iso-contact embedding in $(\S^5, \xi_{std}).$
\end{enumerate}
\end{theorem}

Finally, we discuss iso-contact embeddings of simply-connected contact $5$--manifolds in $(\S^7, \xi_{std}).$ 
In particular, we establish:

\begin{theorem}\label{thm:simply_connected_5-folds}
Let $(M,\xi)$ be a closed simply connected contact $5$--manifold with $w_2(M) = 0$. 
Then, $(M,\xi)$ admits an  iso-contact embedding in $(\S^7, \xi_{std})$ if and
only if $c_1(\xi) = 0.$
\end{theorem}



\section*{Acknowledgment} 

The first author is extremely grateful to Yakov Eliashberg for 
asking various questions regarding iso-contact embeddings which stimulated this work. During 
the course of the development of this article, he provided us  constant encouragement
and support. This work would not have been possible without his critical comments and suggestions. 
We are also thankful to Roger Casals and John Etnyre  for various helpful comments 
and suggestions. We are extremely grateful to Francisco Presas for helping us
 improving the presentation of this  article. 
The first author is also thankful to Simons Foundation for providing
support to travel to Stanford, where a part of work of this project was carried out. The first author
is  thankful to ICTP, Trieste, Italy and Simons Associateship program without which this work
would not have been possible.


\section {Preliminaries}\label{sec:prelim}

In this section, we quickly review notions necessary for the article pertaining 
open books, contact structures and relationship between  them.

\subsection{Open books}
\mbox{}

Let us review few results related to open book decomposition of manifolds.
We first recall the following:

\begin{definition}[Open book decomposition] 
An open book decomposition of a closed oriented  manifold $M$ consists of a co-dimension $2$ oriented 
sub-manifold $B$ with a trivial normal bundle in $M$ and a  locally trivial fibration
$ \pi: M \setminus B \rightarrow \S^1$ such that $\pi^{-1}(\theta)$ is an interior of a
co-dimension $1$ sub-manifold $N_{\theta}$ and $\partial N_{\theta}  = B$, for all $\theta \in \S^1.$
Furthermore, the  normal bundle $\mathcal{N}(B)$ of the sub-manifold $B$ is trivialized 
such that $\pi$ restricted
to $\mathcal{N}(B) \setminus B \rightarrow \S^1$ is given by the angular co-ordinate in $\mathbb{D}^2$--factor. 
\end{definition}

The sub-manifold $B$ is called the  \emph{binding} and  $N_{\theta}$ is called a \emph{page} 
of the open book. We denote the open book decomposition of $M$ by $(M, \mathcal{O}b(B, \pi))$ or sometimes
simply by $\mathcal{O}b(B, \pi).$

Next, we discuss the notion of an \emph{abstract open book decomposition}. To begin with, let us recall
that  the \emph{mapping class group} of a manifold $(\Sigma, \partial \Sigma)$ is the 
group of isotopy classes of diffeomorphisms of $\Sigma$ which are the identity near the boundary 
$\partial \Sigma$. 

\begin{definition}[Mapping torus] Let $\Sigma$ be a manifold with non-empty boundary $\partial \Sigma$. 
Let $\phi$ be an element of the mapping class group of $\Sigma$. By the mapping torus 
$\mathcal{MT}((\Sigma, \partial \Sigma), \phi),$ we mean

$$ \Sigma \times [0,1] / \sim, $$  

where $\sim$  is the equivalence relation identifying $(x, 0)$ with $(\phi(x), 1).$
\end{definition}

Observe that by the definition of $\mathcal{M}T((\Sigma,\partial \Sigma), \phi)),$  there exists a
collar of the boundary $\partial \mathcal{M}T((\Sigma, \partial \Sigma), \phi)$ in 
$\mathcal{M}T((\Sigma, \partial \Sigma), \phi)$ which can be
identified with $(-\epsilon, 0] \times \partial \Sigma \times \S^1$ as the diffeomorphism $\phi$ is 
the identity in a collar $(-\epsilon, 0] \times \partial \Sigma$ of the boundary of $\Sigma.$
We will sometimes denote the mapping torus $\mathcal{M}T((\Sigma, \partial \Sigma), \phi)$ just 
by $\mathcal{M}(\Sigma, \phi).$
We are now in a position to define an abstract open book decomposition.

\begin{definition}[Abstract open book ]
Let $\Sigma$ and $\phi$ as in the previous definition. An abstract open book decomposition of 
$M$ is pair $(\Sigma, \phi)$ such that $M$ is diffeomorphic to 

$$\mathcal{MT}(\Sigma, \phi) \cup_{id} \partial \Sigma \times \mathbb{D}^2, $$

\noindent where $id$ denotes the identity mapping of $\partial \Sigma \times \S^1$

\end{definition}

The map $\phi$ is called the \emph{monodromy} of the open book. We will denote an abstract open book 
decomposition by  $\mathcal{A}ob(\Sigma, \phi).$ Note that
the mapping class $\phi$ uniquely determines $M = \mathcal{A}ob(\Sigma, \phi)$ up to diffeomorphism.

One can easily see that an abstract open book decomposition of $M$ gives an open book decomposition of $M$
up to diffeomorphism and vice versa. Hence, sometimes we will not distinguish between open books and
abstract open books. In particular, we will continue to use the notation $\mathcal{A}ob(\Sigma, \phi)$ to
denote the open book decomposition associated to the abstract open book $\mathcal{A}ob(\Sigma, \phi).$

\begin{example}\label{rmk:standard_ob}
\begin{enumerate}
\item
Notice that $\mathbb{S}^n$ admits an open book decomposition with pages $\mathbb{D}^{n-1}$ and the monodromy the identity
map of $\mathbb{D}^{n-1}.$ We call this open book the \emph{trivial open book}
of $\S^n.$
For more details regarding open books, refer the 
lecture notes \cite{Et} and \cite[Chpt-4.4.2]{Gi}.

\item The manifold $\mathbb{S}^3 \times \mathbb{S}^2$ admits an open book decomposition with pages disk co-tangent bundle  $\mathcal{D}T^*\mathbb{S}^2$ and  
monodromy the identity.
We call this open book decomposition of $\S^3 \times \S^2$  the standard open book decomposition of
$\mathbb{S}^3 \times \mathbb{S}^2.$

\item In \cite{Al}, it was shown that every closed orientable $3$--manifold admits an open book decomposition. 
This result was further generalized to all odd dimensional closed orientable manifold of dimension bigger than $5$ by Quinn in  \cite{Qu}.

\end{enumerate}

\end{example}

Given two abstract open books $M_1^n = \mathcal{A}ob(\Sigma_1, \phi_1)$ and $M_2^n = \mathcal{A}ob(\Sigma_2, \phi_2),$ if we make the band connected sum of pages of 
$\mathcal{A}ob(\Sigma_1, \phi_1)$ and $\mathcal{A}ob(\Sigma_1,
\phi_2),$ then we get an abstract open book decomposition $\mathcal{A}ob(\Sigma_1 \#_b \Sigma_2, \phi_1 \# \phi_2)$ of
$M_1 \# M_2.$ This was first established in  \cite{Ga} by D.Gabai for $3$--manifolds. 

There exists an intimate connection between open books and contact structures on manifolds
discovered by E.Giroux in \cite{Gi}. 
In order to understand this correspondence, we first recall the notion of a contact structure
supported by an open book.

\subsection{Contact manifolds and supporting open books}\label{subsec:open_book_and_contact_structure}
\mbox{}

Giroux in \cite{Gi} introduced the notion of a contact structure supported by an open book.  
We  now recall this notion.

\begin{definition}[Open book supporting a contact form]
 Let $(M, Ker \{ \alpha \})$ be a contact manifold. 
 We say that an open book decomposition $\mathcal{O}b(B, \pi)$ supports
  a contact form $\alpha$ provided:
 \begin{enumerate}
\item The binding $B$ is a contact sub-manifold of $M.$

\item $d \alpha$ is a symplectic form on each page of the open book.

\item The boundary orientation on $B$ coming from the orientation of the pages induced by $(d \alpha)^n$ is
the same as the orientation given by $\alpha|_{B}\wedge (d \alpha|_B)^{n-1}.$ 
\end{enumerate}
\end{definition}

We would like to remark that if $\alpha_1$ and $\alpha_2$ are two contact forms on a contact manifold $M$ 
which are supported by the same open book $\mathcal{O}b(B, \pi),$ then they are isotopic as contact structures.
See, for example, \cite{Ko}.

Let $(M, \xi)$ be a contact manifold. We say that $\xi$ is supported by an open book decomposition $\mathcal{O}b(B,\pi)$ of
$M$ provided there exists a contact $1$--form $\alpha$ inducing the contact structure $\xi$ on $M$ such that
$\alpha$ is supported by $\mathcal{O}b(B, \pi).$

Giroux in \cite{Gi} established a one to one correspondence between the open books up to 
\emph{positive stabilizations} and the  supported contact structures up to isotopy for closed orientable
$3$--manifolds. See the notes \cite{Ko1} by O. van Koert 
and \cite{Et} for more on this. The purpose of the next subsection is to recall a few notions 
and results associated to Giroux's correspondence.

\subsection{Contact abstract open book and Giroux's correspondence}\label{subsec:Giroux_correspondence}
\mbox{}

We begin this subsection by  recalling  the notion of the
\emph{Generalized Dehn twist}. This notion is necessary to understand the notion of
positive stabilization. This notion was first introduced in \cite{Se1}. See also \cite{Se2}.

\begin{definition}[Generalized Dehn twist]

Consider,

$$T^*\mathbb{S}^n = \{(x,y) \in \mathbb{R}^{n+1} \times \mathbb{R}^{n+1} | \hspace{0.2cm} x.y = 0,  ||y|| = 1  \}.$$

Define a diffeomorphism $\tau$ of $T^*\mathbb{S}^n$ as follows:

\begin{equation*}
 \tau(x,y) =  \begin{pmatrix}
    \cos g (y)& |y|^{-1} \sin g(y) \\
    -|y| \sin g (y)&\cos g(y)
  \end{pmatrix}
  \begin{pmatrix}
    x\\
    y
  \end{pmatrix}
\end{equation*}

where, $g$ is a function of $y$ which is the identity near $0$ and is zero outside a compact set containing $0.$
The diffeomorphism $\tau$ is called the Generalized Dehn twist while
$\tau^{-1}$ is called the negative generalized Dehn twist.

\end{definition}

It is relatively easy to check that $\tau$ is a compactly 
supported symplectomorphism of $T^*\mathbb{S}^n.$ Furthermore,  $\tau$ can be isotoped to 
a symplectomorphism which is compactly supported in an arbitrary small neighborhood of the zero section of 
$T^*\mathbb{S}^n.$ This, in particular, implies that $\tau$ and $\tau^{-1}$ 
can be regarded as diffeomorphisms of the disk co-tangent bundle $\mathcal{D}T^*\S^n.$ We refer to \cite[page-186]{MS} and the notes~\cite{Ko} for more details. See also \cite{KN} for a nice
exposition on how to produce a compactly supported Generalized Dehn twist.

Next, we discuss the notion of a \emph{contact abstract open book.} We refer to \cite[Section--2]{Ko} for a
more detailed description of this. 

Let $(\Sigma, d \lambda)$ be a Weinstein manifold and $\phi$  an exact  symplectomorphism
of $\Sigma$ which is the identity near the boundary of $\Sigma.$ Giroux  generalized the construction 
of W. Thurston and H. Winkelnkemper given in \cite{TW} to produce a contact form on the manifold with
open book
$\mathcal{A}ob(\Sigma, \phi)$ such that the contact form is  supported by the   
open book $\mathcal{A}ob(\Sigma, \phi)$ in the sense explained in 
Subsection~\ref{subsec:open_book_and_contact_structure}. 
We will generally denote this contact form by $\alpha_{(\Sigma, \phi)}.$
See  lecture notes by O. van Koert \cite{Ko} for the details of this
construction. See also \cite{GM}.
We call this contact manifold a  \emph{contact abstract open book}. 

In this article, unless stated otherwise,
whenever we talk of a contact structure $\xi$, supported by an abstract open book $\mathcal{A}ob(\Sigma, \phi),$
we will always mean that $\Sigma$ is a Weinstein manifold and $\phi$  an exact symplectomorphism of $\Sigma$
which when restricted to a collar of its boundary is the identity and the contact structure $\xi$ is contactomorphic to $Ker \{\alpha_{(\Sigma, \phi)}\}$ described
earlier.

\begin{example}\label{exm:examples_of_ob}\mbox{}

\begin{enumerate}
\item
Let $\mathbb{D}^{2n}$ denote the unit $2n$--disk in $\mathbb{R}^{2n}$. Let
$\lambda_{std}$ denote the canonical $1$--form on $\mathbb{D}^{2n}$ given by $\displaystyle\sum_{i=1}^n x_i dy_i
-y_i dx_i$ which induces the standard symplectic structure on $\mathbb{D}^{2n}.$ The standard contact
sphere $(\mathbb{S}^{2n+1}, \xi_{std})$ is contactomorphic to the open 
book $(\mathcal{A}ob(\mathbb{D}^{2n}, id), Ker \{\alpha_{(\mathbb{D}^{2n},id)}\}).$ 

\item Consider $\mathcal{D}T^*\mathbb{S}^n$, the unit disk bundle associated to the co-tangent bundle of $\S^n$ and 
the Generalized Dehn twist $\tau$ on $\mathcal{D}T^*\mathbb{S}^n$.
 It is well known that the contact abstract open book 
 $\mathcal{D}\mathcal{A}ob(T^*\mathbb{S}^n, \tau)$
is contactomorphic to the standard contact sphere $(\mathbb{S}^{2n+1}, \xi_{std}).$

\item The contact abstract open book
$\mathcal{A}ob(\mathcal{D}T^*\mathbb{S}^n, \tau^{-1})$ induces an overtwisted contact structure on $\S^{2n+1}.$ 
This is clearly discussed for  $3$--manifolds in \cite{KN1}. In general, this follows
from \cite{CMP}.  We will denote this overtwisted contact structure by $\xi_{stot}.$ 
We will denote a contact $1$--form inducing the contact structure $\xi_{stot}$ by $\alpha_{stot}.$
\end{enumerate}

\end{example}

We now define the notion of  a  \emph{generalized contact abstract open book}:

\begin{definition}[Generalized contact abstract open book]\label{def:generalized_abstract_open_book}

Let  $((W, \partial W), d \lambda)$  be a Weinstein cobordism with a connected convex boundary $M$. Let 
$\phi$ be a symplectomorphism of $(W, d \lambda)$ which is the identity in a small collar of
the boundary of $W.$  Consider the quotient manifold $N$ defined as:

$$N = \mathcal{MT}(W, \phi) \cup_{id} M \times \mathbb{D}^2, $$

\noindent notice that $N$ admits a contact structure analogous to the one discussed earlier for 
the contact abstract  open book. 
We call $N$ a generalized contact abstract open book with the binding $M,$ page $W$ and 
the monodromy $\phi.$ 
\end{definition}

By a slight abuse of notation, we will use  the same notation $\mathcal{A}ob(W, \phi)$ for 
the generalized contact abstract 
open book as well.  By a Weinstein manifold or Weinstein domain, we will always mean a Weinstein cobordism with
an empty concave boundary and a connected convex boundary. Note that  whenever the Weinstein cobordism associated to
a generalized contact abstract open book is a Weinstein manifold, we get usual contact abstract open book.

\begin{definition}[Generalized contact abstract
connected sum]\label{def:contact_band_sum} \mbox{}

Let $(\mathcal{A}ob(W_1, \phi_1), \alpha_{(W_1, \phi_1)})$
and $(\mathcal{A}ob(W_2, \phi_2), \alpha_{(W_2, \phi_2)}) $ be two generalized contact abstract open books. 
Observe that we can perform the band connected sum
$W_1 \#_b W_2$ of
$W_1$ and $W_2$ along their connected convex boundaries to 
produce a new Weinstein cobordism $W_1 \#_b W_2$ with
connected convex boundary $\partial W_1 \# \partial W_2.$ Let  
$\mathcal{A}ob(W_1, \phi_1) \# \mathcal{A}ob(W_2, \phi_2)$ be the generalized abstract open book  obtained by performing the band connected sum of their pages along convex boundaries. 

Since the page of the generalized abstract open book
are Weinstein cobordism  $W_1 \#_b W_2$ with connected
convex boundary $\partial W_1 \# \partial W_2,$ it is clear that  this generalized abstract open book   carries a natural contact structure supported by the  generalized open book having pages $W_1 \#_b W_2$ and the monodromy $\phi_1 \# \phi_2.$  This contact structure will be denoted by $Ker \{\alpha_{(W_1 \#_b W_2, \phi_1 \# \phi_2)} \}.$
We call this contact manifold   the  \emph{generalized contact abstract  connected sum}. 
\end{definition}

\begin{remark}\label{rem:contact_band_sum}\mbox{}
\begin{enumerate}
\item Observe that the binding of a generalized contact abstract band connected sum is the connected sum of the
bindings of the generalized contact abstract open books.

\item When $W_1$ and $W_2$ are Weinstein manifolds, then the generalized contact
abstract connected sum 
 $\mathcal{A}ob(W_1, \phi_1) \#_b \mathcal{A}ob(W_2, \phi_2)$ is the
contact connected sum of $\mathcal{A}ob(W_1, \phi_1)$ and 
$\mathcal{A}ob(W_2, \phi_2).$ The contact structure 
$Ker\{\alpha_{(W_1 \#_b W_2, \phi_1 \# \phi_2)}  \alpha\}$ is supported by the
open book with pages $W_1 \#_b W_2$ and the monodromy $\phi_1 \# \phi_2.$ 

\item We will sometime  use the notation 
$\mathcal{A}ob(W_1 \#_b W_2, \phi_1 \# \phi_2)$ to denote the generalized
abstract connected sum  
$\mathcal{A}ob(W_1, \phi_1) \# \mathcal{A}ob(W_2, \phi_2).$ This notation
will be used  to emphasis the abstract open book decomposition
of $\mathcal{A}ob(W_1, \phi_1) \# \mathcal{A}ob(W_2, \phi_2).$

\end{enumerate}

\end{remark}

%
%
%
%
%

\subsection{Iso-contact open book embeddings}

In this sub-section, we  discuss the notion of  \emph{iso-contact open book embeddings}. For more on open book embeddings,
refer \cite{EL} and \cite{PPS}.

\begin{definition}\label{def:contact_abstract_open_book_embedding}

Let $M = \mathcal{A}ob(\Sigma, \phi)$ and $N = \mathcal{A}ob(W, \Psi)$ be two generalized contact abstract open books.
Let $F:M \rightarrow N$ be a proper iso-contact embedding of $M$ in $N.$ We say that this embedding 
is a contact abstract open book embedding,  provided the following diagram commutes:

$$ \xymatrix{
 \mathcal{M}T(\Sigma, \phi)   \ar[rd]^{\pi_1} \ar@{^{(}->}[r]^F &  
 \mathcal{M}T(W, \Psi) \ar[d]^{\pi_2}  \\
 & \S^1 .
} $$

Here, $\pi_1:\mathcal{M}T(\Sigma, \phi) \rightarrow \S^1$ and 
$\pi_2: \mathcal{M}T(W, \Psi) \rightarrow \S^1$ are the natural projections associated to the mapping tori.

\end{definition}

We end this section by  establishing  a proposition.  The proposition, in particular, establishes  that 
if $(M_1, \xi_{M_1})$ iso-contact embeds in $(N_1, Ker\{\alpha_1\})$ and $(M_2, \xi_{M_2})$ iso-contact embeds in
$(N_2, Ker \{\alpha_2\})$, then the \emph{contact connected sum} $(M_1 \# M_2, \xi_{M_1} \# \xi_{M_2})$ iso-contact 
embeds in the contact connected sum 
$(N_1 \# N_2, Ker\{\alpha_1\} \# Ker\{\alpha_2\}) = (N_1 \# N_2, Ker\{ \alpha_1 \# \alpha_2 \} ).$ This was already proved 
by J.Etnyre and R. Fukuwara in \cite{EF}. 

\begin{proposition}\label{pro:connect_sum_embedding}
 
 If a contact abstract open book $(\mathcal{A}ob(\Sigma_i^{2n-2}, \phi_i), \eta_i)$ iso-contact open book embeds in a 
 generalized contact abstract open book 
 $(\mathcal{A}ob(W_i^{2n}, \Psi_i), \xi_i)$ for $i = 1,2,$ 
 then the contact abstract connected sum
 $(\mathcal{A}ob(\Sigma_1, \phi_1) \# \mathcal{A}ob(\Sigma_2, \phi_2), \eta_1 \# \eta_2)$ iso-contact
 open book embeds in the generalized  contact abstract  connected sum 
 $(\mathcal{A}ob(W_1, \Psi_1) \# \mathcal{A}ob(W_2, \Psi_2), \xi_1 \# \xi_2).$

 Furthermore, if $\Sigma_i$ is contained in an arbitrary small  collar of the convex boundary $M_i$ of 
 $\partial W_i, $ then we can ensure that 
 the page $\Sigma_1 \#_b \Sigma_2$ of 
 $\mathcal{A}ob(\Sigma_1 \#_b \Sigma_2, \phi_1 \# \phi_2)$ 
 is contained in an arbitrary  small collar of the convex boundary of the page of  
 $\mathcal{A}ob(W_1 \#_b W_2, \Psi_1 \# \Psi_2) .$

\end{proposition}

\begin{proof}
 First of all notice that since the  band connected sum of $W_1$ with $W_2$ can be
 regarded as adding a $1$--handle to $W_1 \sqcup W_2,$ 
 we can perform the band connected sum of $W_1$ with $W_2$ along their convex boundaries in such way that
 the band connected sum of $\Sigma_1$ with $\Sigma_2$ properly symplectically embeds in $W_1 \#_b W_2.$ 
 To achieve this, notice that in order to perform the band connected sum, we need to fix a small Darboux ball
 $U_1$ around a point $p_1$ in
 $\partial W_1$ and a small Darboux  ball $U_2$ around a point $p_2$  in $\partial W_2.$ 
 We fix these balls  in such a way that they 
 restrict to Darboux balls $\widetilde{U}_i$ containing the point
 $p_i$ in $\partial \Sigma_i,$ for each $i = 1,2.$ Now, if we perform the band connected
 sum of $W_1$ with $W_2,$ we get an induced band connected sum of $\Sigma_1$ with $\Sigma_2$ which is contained
 in $W_1 \#_b W_2.$ 
 
 Observe that we have not yet achieved the second property. In order to achieve this,
 we first  observe  that
 $\Sigma_1 \#_b \Sigma_2 \subset W_1 \#_b W_2$ can be made disjoint from the core of the $1$--handle $B$
 associated to $W_1 \#_b W_2$ by a sufficiently small $C^{\infty}$ perturbation whose support
 is contained in a small tubular neighborhood of  $ B \cap \Sigma_1 \#_b \Sigma_2 \subset W_1 \#_b W_2.$
 See Figure~\ref{fig:nbhd_point_connect_sum} for a pictorial description.
  
 Let $\epsilon_1$ be such that $\Sigma_1$ is contained in the symplectic collar 
 $([0, \epsilon_1 ] \times \partial W_1, d( e^t \alpha_1))$
 of $\partial W_1$, where  $e^t \alpha_1$ is the Liouville $1$--form on the symplectic collar
 of the convex boundary $\partial W_1$.

 Let $\epsilon_2$ be such that $\Sigma_2$ is contained in the symplectic collar 
 $([0, \epsilon_2] \times \partial W_2, d( e^t \alpha_2))$
 of $\partial W_1$, where  $e^t \alpha_2$ is the Liouville $1$--form on the symplectic collar of
 the convex boundary $\partial W_2.$

 Let us denote by $B = \D^{2n-1}( \delta)  \times \D(1)$  the band of length $1$ and radius $\delta$ used
 in the band connected sum $W_1 \#_b W_2.$ Clearly, by the construction $\widetilde{B} = \D^{2n-3}(\delta) \times \D(1)$
 is the band  associated to the induced band connected sum $\Sigma_1 \#_b \Sigma_2.$
 
 Let $A_{\delta}$ denote the annulus $[\frac{\delta}{10}, \delta] \times \S^{2n-2} \times \D^1.$
Notice that the part of the boundary of the band $B$ corresponding to $\D^{2n-1} \times \partial \D(1)$ 
can be assumed to have the symplectic collar $A_{\delta}.$ Hence, if  the $C^{\infty}$--perturbation that
we perform in order to make $\Sigma_1 \# \Sigma_2$ disjoint from the core of $1$-- handle is done such that
perturbed $\Sigma_1 \#_b \Sigma_2$ is contained in $A_{\delta}$ and the support of the perturbation is
contained in the complement of 
the annulus  $ [\frac{9 \delta}{10}, \delta] \times \S^{2n-1} \times \mathbb{D}^1, $ then  
the perturbed band $\widetilde{B}$ associated to $\Sigma_1 \#_b \Sigma_2$ is contained in $A_{\delta}$ and 
its intersection with the boundary of the annulus $A_{\delta}$ is the same as the intersection of unperturbed
$\widetilde{B}.$ Observe that such a perturbation is always possible. 

Next, choose $\delta$ such that $\delta < min \hspace{0.1cm} \{\epsilon_1, \epsilon_2\}.$ 
Observe that for this choice of $\delta $ the perturbed $\Sigma_1 \#_b \Sigma_2$ lies in a small 
symplectic neighborhood of $\partial W_1 \# \partial W_2$ as claimed. 
See Figure~\ref{fig:band_sum_near_boundary}.

 Finally, observe that since the
 symplectomorphisms $\Psi_1$ and $\Psi_2$ are the identity in suitable collars of the boundaries of
 $W_1$ and $W_2$ respectively, the symplectomorphism
 $\Psi_1 \# \Psi_2$ naturally induces the symplectomorphism $\phi_1 \# \phi_2$ on the 
 symplectically embedded $\Sigma_1 \#_b \Sigma_2 \subset W_1 \# W_2$ that we just described.

 
 This establishes  the proposition. 
\end{proof}

\begin{figure}[ht]
\begin{center}
\psfrag{W1}[c][t]{$ \partial W_i$}
\psfrag{W2}[c][t]{$\partial W_i$}
\psfrag{S1}[c][t]{$\partial \Sigma_i$}
\psfrag{p}[c][b]{$p_i$}
\includegraphics[width=10cm,height=4cm]{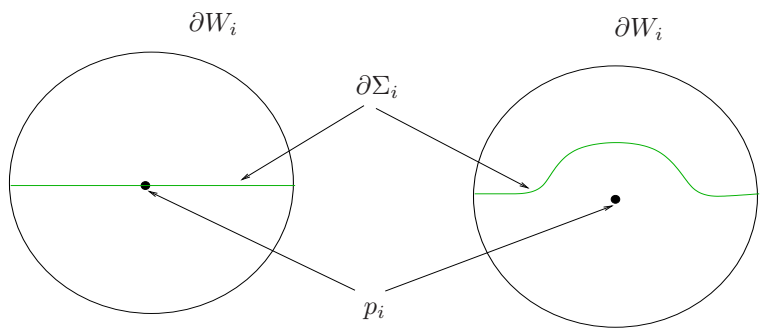}
\caption{The figure depicts a small Darboux neighborhood  of the attaching sphere $p_i$ 
contained in  $\partial \Sigma_i \subset \partial W_i$ used in performing
the band connected sum $W_1 \#_b W_2.$ The picture of on the left depicts the embedding of the
Darboux ball  of $\Sigma_i$ contained   the unperturbed $\Sigma_1 \#_b \Sigma_2$, while the
the picture on the right depicts the neighborhood after a sufficiently  small perturbation.}
\label{fig:nbhd_point_connect_sum}
\end{center}
\end{figure}

 \begin{figure}[ht]
 \begin{center}
 \psfrag{W1}[c][t]{$ [0, \epsilon] \times  \partial W_1$}
 \psfrag{W2}[c][t]{$ [0, \epsilon] \times \partial W_i$}
 \psfrag{S1}[c][b]{$\Sigma_1$}
 \psfrag{S2}[c][b]{$\Sigma_2 $}
 \psfrag{S!+S2A}[c][b]{$\Sigma_1 \#_b \Sigma_2$}
 \psfrag{W1+W2}[c][b]{$\mathcal{N}(\partial (W_1 \#_b W_2)) $}
 \psfrag{A}[c][t]{$B$}
 \psfrag{B}[c][t]{$\partial (W_1 \#b W_2)$}
 \includegraphics[width=10cm,height=4cm]{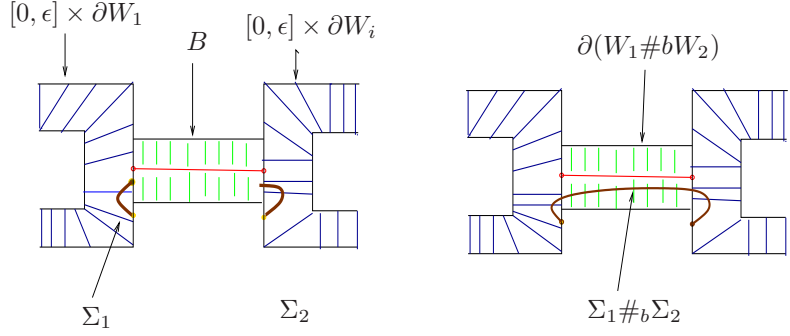}
 \caption{The figure on the left depicts a collar of $\partial W_i$ containing $\Sigma_i$ for each $i$
together with the band $B.$ The red line at the center of the band $B$ is the core of the attaching band. 
The figure on the right depicts $\Sigma_1 \# \Sigma_2$ embedded in $W_1 \#_b W_2$ close to the boundary 
$\partial (W_1 \#_b W_2) $ and disjoint from the core of the $1$-handle $B$. The green region in both figures depicts the annulus $A(\delta).$}
 \label{fig:band_sum_near_boundary}
 \end{center}
 \end{figure}

\section{Proof of the Theorem~\ref{thm:h-principle}}
\label{sec:local_flexibility}

The purpose of this section is to establish the Theorem~\ref{thm:h-principle}. There are three 
steps in establishing the Theorem~\ref{thm:h-principle}. We first mention the first  two
steps in the form of the Proposition~\ref{pro:h-principle_weak} and the Proposition~\ref{pro:ot_sphere_in_stand_sphere}.
We give proofs of these propositions in Section~\ref{sec:key_pro}.

In order to state the Proposition~\ref{pro:h-principle_weak}, we need to introduce the following notation. 
Let $(M^{2n+1}, \xi)$ be a contact manifold.  The contact structure obtained by the contact connected
sum of $(M, \xi)$ with the standard overtwisted sphere $(\S^{2n+1}, \xi_{stot})$ will be 
denoted by $\xi^{stot}.$ Notice that if  $\xi$ is supported by an open book decomposition 
$\mathcal{A}ob(\Sigma, \phi),$ then $\xi^{stot}$ is supported by the   open book 
$\mathcal{A}ob(\Sigma \#_b \mathcal{D} T^*\S^n, \phi \# \tau^{-1}).$ This follows from \cite{CMP}. See, the third example in Examples~\ref{exm:examples_of_ob}.

\begin{proposition}\label{pro:h-principle_weak}
Let $M^{2n-1}$ be a closed smooth manifold. Let $\xi$ be a contact structure on $M.$
Suppose that $(M, \xi^{stot})$ admits an iso-contact embedding in a contact manifold 
$(N^{2n+1}, \xi_N)$ with the trivial normal bundle, then  $(M, \xi)$ also admits an iso-contact embedding in 
$(N, \xi_N).$
\end{proposition}

 This is the key step and its proof is divided in to several smaller steps. As mentioned earlier, we will
establish each step in Section~\ref{sec:key_pro}.

The second step is to establish the following:

\begin{proposition}\label{pro:ot_sphere_in_stand_sphere}
 There exists an iso-contact embedding of $(\mathbb{S}^{2n-1}, \xi_{stot})$ in the standard contact sphere
 $(\mathbb{S}^{2n+1}, \xi_{std}).$
\end{proposition}

Finally, the third step is to establish the following:

\begin{proposition}\label{pro:embedding_of_contact_sum_with_std_ot}
 Let $(M^{2n-1}, \xi_1)$ be a contact manifold.  If $(M, \xi_1)$ admits an
 iso-contact embedding in a contact manifold $(N, \xi_2),$ then there exists an iso-contact embedding of $(M, \xi_1^{stot})$ in the contact manifold $(N, \xi_2).$
\end{proposition}

\begin{proof}

Observe  that it follows from the Proposition~\ref{pro:connect_sum_embedding} that if 
$(M_1, \xi_1)$ iso-contact embeds in $(N_1, \eta_1)$ and $(M_2, \xi_2)$ iso-contact 
$(N_2, \eta_2)$, then $(M_1 \# M_2, \xi_1 \# \xi_2)$ iso-contact embeds in 
$(N_1 \# N_2, \eta_1 \# \eta_2).$

We know from the Proposition~\ref{pro:ot_sphere_in_stand_sphere} that  there is an iso-contact embedding of 
$(\mathbb{S}^{2n-1}, \xi_{stot})$ in the
standard contact  sphere $(\mathbb{S}^{2n+1}, \xi_{std}).$ 
This implies that $(M \# \mathbb{S}^{2n-1}, \xi_1 \# \xi_{stot})$ admits an iso-contact 
embedding in $(N \# \mathbb{S}^{2n+1}, \xi_2 \# \xi_{std}).$

Next, note that  the contact manifold $(N \# \mathbb{S}^{2n+1}, \xi_2 \# \xi_{std})$ is 
contactomorphic  to $(N, \xi_2).$ Hence, we have an iso-contact embedding
of $(M, \xi_1^{stot}) = (M \# \S^{2n-1}, \xi_1 \# \xi_{stot})$ in the 
contact manifold $(N, \xi_2).$ 

\end{proof}

Let us now discuss how these three steps imply  the Theorem~\ref{thm:h-principle}.

\begin{proof}[ Proof of the Theorem~\ref{thm:h-principle}] \mbox{}

When the co-dimension of the embedding of $M$ in $N$ is bigger than or equal to $4,$ the 
theorem was already established by Gromov in \cite{Gr}. Hence, from 
now on, we assume that the co-dimension of $M$ in $N$ is $2.$

First of all,  since an overtwisted contact structure is unique in its homotopy class of almost contact structures,
we get that the contact structures $(M, \xi_M^{ot} \# \xi_{stot})$ and $(M, \xi_M^{stot})$ are contactomorphic.
Hence,  it follows from the Proposition~\ref{pro:embedding_of_contact_sum_with_std_ot}  that there is an iso-contact embedding of $(M, \xi_M^{stot})$
in $(N, \xi_N)$ with the trivial normal bundle.  
It follows from the proposition~\ref{pro:h-principle_weak}  that there is an iso-contact embedding of $(M, \xi_M)$
in $(N, \xi_N)$ as required.

\end{proof}

Let us now discuss how the Proposition~\ref{cor:h-principle_for_embedding_in_spheres} follows from the Theorem~\ref{thm:h-principle}.

\begin{proof}[Proof of the Proposition~\ref{cor:h-principle_for_embedding_in_spheres}]\mbox{}

Since the Euler class of the normal bundle of any embedded closed orientable manifold $M$ in 
$\S^k$ has to be zero, we get that the manifold $M^{2n-1}$ admits an embedding 
in $\S^{2n+1}$ with the trivial normal bundle. 

Next, assume that $M$ admits an embedding in $(\S^{2n+1}, \xi_{std})$ such that the
induced contact structure $\xi$ is homotopic to $\xi_M.$  
The Proposition~\ref{pro:ot_sphere_in_stand_sphere} implies
that there exits an iso-contact embedding of $(\S^{2n-1},\xi_{stot})$ in
$(\S^{2n+1}, \xi_{std}).$  Hence, it follows from the Proposition~\ref{pro:connect_sum_embedding} that there exists an iso-contact embedding of
$(M, \xi \# \xi_{stot})$ in $(\S^{2n+1}, \xi_{std}).$ 
Now, since the overtwisted contact structures $\xi \# \xi_{stot}$ and 
the $\xi_M \# \xi_{stot}$ are homotopic as almost contact structures, by the
uniqueness of an overtwisted contact structure in a given homotopy class of 
almost contact structures, we get that there is an iso-contact embedding 
of $(M, \xi_M \# \xi_{stot}) = (M, \xi_M^{stot})$ in the standard contact sphere.

The Proposition~\ref{pro:h-principle_weak} now implies that there is an
iso-contact embedding of $(M, \xi_M)$ in the standard contact sphere as claimed.
This completes our argument.

\end{proof}

The next section is devoted to establish the proof of the Proposition~\ref{pro:h-principle_weak}. 
Using the techniques developed to establish the Proposition~\ref{pro:h-principle_weak}, we
will also establish the Proposition~\ref{pro:ot_sphere_in_stand_sphere}.

\section{Proofs of  Proposition~\ref{pro:h-principle_weak} and Proposition~\ref{pro:ot_sphere_in_stand_sphere}}
\label{sec:key_pro}

In order to prove the
Proposition~\ref{pro:h-principle_weak}, we would like to think of $M \times \mathbb{D}^2_{\varepsilon}$ as an abstract open book. This
abstract open book is a special case of a generalized abstract open book.  Since we will need it
time and again,  we introduce a special terminology for it. 

\begin{definition}[$\varepsilon$--partial open book]
Let $W$ be the product Weinstein cobordism $([a,\varepsilon] \times M, d (e^{-r} \alpha)),$ where $\alpha$ be a contact form
on $M$ and $\{a\} \times M$ is the convex boundary. 
Consider the mapping torus $\mathcal{M}T(W, id) = \S^1 \times ([a,\varepsilon] \times M)$ with the
contact form $ \varepsilon d \theta + e^{-r} \alpha.$  Consider the  contact abstract open book
$(\mathcal{A}ob(W, id), \alpha_{(W, id)}) = (M \times \mathbb{D}^2_{\varepsilon}, Ker\{h_1(r) \alpha + h_2(r) d \theta \} )$  
constructed using the pair of functions $h_1$ and $h_2$  as depicted in Figure~\ref{fig:h_1_and_h_2} and satisfying the following properties:

\begin{enumerate}
 \item $h_1(r) > 0$,  decreasing and $h_1'(0) = 0$ and $h_1'(r) <0$ for every $r \in (0,b).$
 \item $h_1(r) = e^{-r}$ near $\varepsilon.$ 
 \item $h_2(r) = r^2$ near $r = 0$ and $h_2$ is non-decreasing and $h_2(r)$ is the constant $\varepsilon$  near $\varepsilon.$
\item $h_2^{\prime}(r) h_1(r) - h_1'(r)h_2(r)$ is always positive.
\end{enumerate}

This contact open book is called an $\varepsilon$--partial open book associated to  $(M, Ker\{\alpha \}).$
\end{definition}

\begin{figure}[ht]
\begin{center}
\psfrag{h1}{$h_1$}
\psfrag{h2}{$h_2$}
\includegraphics[width=10cm,height=4cm]{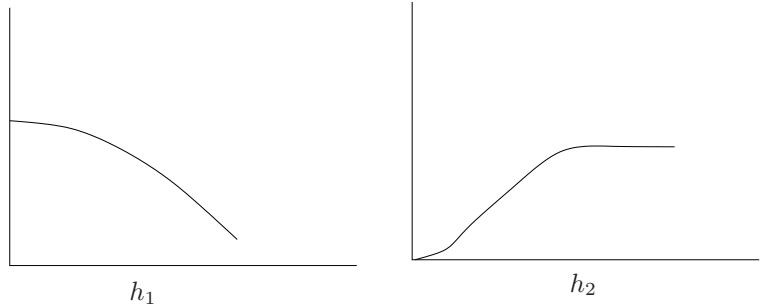}
\caption{The figure depicts the graphs of the functions $h_1$ and $h_2.$ 
These functions are also used in the construction of contact abstract open 
book in \cite{Ko}.}
\label{fig:h_1_and_h_2}
\end{center}
\end{figure}

\begin{remark}\label{rmk:std_region_pob}
\begin{enumerate}
\item The form $h_1(r) \alpha + h_2(r) d \theta$ on the $\varepsilon$--partial open
book is supported by the $\varepsilon$--partial open book.

 \item The region of an $\varepsilon$--partial open book where the form is given by $\varepsilon d \theta + e^{-r} \alpha$ 
will be referred as the \emph{standard region} associated to the partial open book. 

\item By changing the
co-ordinates $s = -r$ we will denote the standard region by $S^1 \times 
[-\varepsilon, -a] \times M$ for some $-a \in  (-\varepsilon, 0).$ The contact form in the standard region will be described by the formula $\varepsilon d \theta + e^s \alpha.$
%
%

\end{enumerate}
\end{remark}

In the next sub-section, we establish the Proposition~\ref{pro:h-principle_weak}.

\subsection{Proof of Proposition~\ref{pro:h-principle_weak}}\mbox{}

The proof of the Proposition~\ref{pro:h-principle_weak} can be divided into three steps.

The first step, which readily follows from neighborhoods of contact sub-manifolds discussed \cite[Theorem:2.5.15]{Ge}
is stated as the following:

\begin{lemma}\cite[Theorem:2.5.15]{Ge}\label{lem:abstract_darboux}
 Let $(N, \xi_N)$ be a contact manifold. Let $(M, \xi_M)$ be a contact sub-manifold of $(N, \xi_N)$ with the
 trivial normal bundle. If $Ker \{\alpha \}$ is contactomorphic to  $\xi_M$ on $M,$ then there exists an 
 $\varepsilon_0$--positive such that there is an iso-contact embedding of an $\varepsilon$--partial open book associated
 to $(M, Ker \{\alpha\})$ in $(N, \xi_N)$ for  every $\varepsilon$ smaller than $\varepsilon_0.$
 \end{lemma}

We now state the next two steps in the form of lemmas whose proofs we will provide in the subsequent sub-sections.  

The second step is to establish is the following:

\begin{lemma}\label{lem:M_in_M_times_D^2}
Let $(M,Ker\{ \alpha \}) = (\mathcal{A}ob(\Sigma, \phi), Ker \{\alpha_{(\Sigma, \phi)}\})$. Let
$\varepsilon_0 >0$ be given. 
There exists
a contact abstract  open book embedding $F$ of $(M, Ker \{\alpha \})$ in $\varepsilon_0$--partial open book 
$M \times \mathbb{D}^2_{\varepsilon_0}.$  

In particular, this implies the following:

\begin{enumerate}
\item
  $F$ is constructed such that if $\S^1 \times [-\varepsilon_0, -a] \times M$ 
is  the standard region associated to the $\varepsilon_0$--partial open book for some $0 < a < \varepsilon_0,$ then
the following diagram commutes:

$$ \xymatrix{
 \mathcal{M}T(\Sigma, \phi)  \ar[rd]^{\pi_1} \ar@{^{(}->}[r]^F &  
  \S^1 \times [-\varepsilon_0, -a] \times M \ar[d]^{\pi}  \\
 & \S^1 .
} $$ 

\item

The   pull-back under $F$ of the $1$--form $ \varepsilon_0 d \theta + e^s \alpha$ induces the contact structure  $Ker \{\alpha\}$ restricted to $\mathcal{M}T(\Sigma, \phi).$

\end{enumerate}

\end{lemma}

The third step is to establish the following:

\begin{lemma}\label{lem:embd_std_sphere_in_nbhd_stot_sphere}
  For every $\varepsilon > 0,$ there exists an contact abstract open
 book embedding of
 the standard contact sphere $(\mathbb{S}^{2n-1}, \xi_{std})$ in the $\varepsilon$--partial open book associated to 
 $(\mathbb{S}^{2n-1}, Ker \{\alpha_{stot} \}),$ where the standard contact sphere is regarded as an
 abstract open book with pages the standard symplectic $(2n-2)$--disc and monodromy the identity.

In particular, this implies the following:

\begin{enumerate}
\item
  $F$ is constructed such that if $\S^1 \times [-\varepsilon_0, -a] \times \S^{2n-1}$ for some
$ 0 < a < \varepsilon_0$,
is  the standard region associated to the $\varepsilon$--partial open book, then
the following diagram commutes:

$$ \xymatrix{
 \mathcal{M}T(\mathbb{D}^{2n-2}, id)  \ar[rd]^{\pi_1} \ar@{^{(}->}[r]^F &  
  \S^1 \times [-\varepsilon_0, -a] \times \S^{2n-1} \ar[d]^{\pi}  \\
 & \S^1 .
} $$ 

\item

The  pull-back under $F$ of the $1$--form $ \varepsilon d \theta + e^s \alpha_{stot}$ induces 
the contact structure $\alpha_{std}$ restricted to $\mathcal{M}T(\mathbb{D}^{2n-2}, id).$

\end{enumerate}

\end{lemma}

Now that we have clearly stated all three steps needed for the proof of the 
Proposition~\ref{pro:h-principle_weak}  in the form of the Lemmas~\ref{lem:abstract_darboux}, \ref{lem:M_in_M_times_D^2}
and ~\ref{lem:embd_std_sphere_in_nbhd_stot_sphere}, we give the proof of the 
Proposition~\ref{pro:h-principle_weak} assuming the   Lemmas~\ref{lem:M_in_M_times_D^2} 
and ~\ref{lem:embd_std_sphere_in_nbhd_stot_sphere}.

\begin{proof}[Proof of the Proposition~\ref{pro:h-principle_weak}]\mbox{}

We can assume that $(M, \xi)$ is an abstract open book 
$(\mathcal{A}ob(\Sigma, \phi),Ker \{\alpha_{(\Sigma, \phi)}\}).$

We first notice that the Lemma~\ref{lem:abstract_darboux} implies that given an $\varepsilon >0,$
it is sufficient to 
iso-contact embed  $(M, \xi)$ in  the $\varepsilon$--partial open book associated to the 
contact manifold $(M \# \mathbb{S}^{2n-1}, \xi \# \xi_{stot}) = (M ,\xi^{stot}).$

Let $\Sigma_{\theta} = \pi_1^{-1} ( \theta)$, where $\pi_1: \mathcal{M}T(\Sigma, \phi) \rightarrow \S^1$ is
the fibration associated to $\mathcal{A}ob(\Sigma, \phi).$
Now, consider the standard region $\S^1 \times [- \varepsilon, - \varepsilon_0 + \delta] \times M$ 
  of  the
$\varepsilon$--partial open book associated to 
$(M \# S^{2n-1}, Ker \{ \beta = \alpha_{(\Sigma, \phi)} \# \alpha_{stot} \}),$
where the contact form is given by $ \varepsilon d \theta + e^{-r} \beta.$

Observe that if there exists 
a proper symplectic embedding of the band connected sum  $\Sigma_{\theta} \#_b 
\mathbb{D}^{2n-2}$ in 
$\{\theta\} \times (-\varepsilon_0, - \varepsilon_0 + \delta] \times \left( M\# \mathbb{S}^{2n-1}\right)$ for any $\varepsilon_0,$ 
then there exists an iso-contact embedding of
the contact $\mathcal{A}ob(\Sigma \#_b \mathbb{D}^{2n-2}, Ker \{\alpha_{(\Sigma \#_b \mathbb{D}^{2n-2}, \phi \# id)} \})$
in the $\varepsilon$--partial open book associated to $(M \# \mathbb{S}^{2n-1}, \xi \# \xi_{stot}).$


By the Lemma~\ref{lem:M_in_M_times_D^2}, there exists a contact abstract open book embedding of the contact
manifold \\
$(M, Ker \{\alpha \}) = (\mathcal{A}ob(\Sigma, \phi),Ker \{\alpha_{(\Sigma, \phi)}\})$ in the 
$\varepsilon$--partial open book associated to $(M, Ker \{\alpha \}).$  Moreover, the mapping tours
$\mathcal{M}T(\Sigma, \phi)$ is properly embedded close to the convex boundary times $\S^1$  in the 
standard region of the $\varepsilon$--partial open book associated to $(M, Ker \{\alpha \}).$ Furthermore,
the pull-back of the form $\varepsilon d \theta  + e^s \alpha$ induces the contact structure 
$Ker \{\alpha \}$ restricted to $\mathcal{M}T(\Sigma, \phi).$

Also, by the Lemma~\ref{lem:embd_std_sphere_in_nbhd_stot_sphere}, there exists an iso-contact abstract open book
embedding of the standard sphere 
$(\mathbb{S}^{2n-1}, Ker \{\alpha_{std}\}) = (\mathcal{A}ob(\mathbb{D}^{2n-2},id),Ker \{\alpha_{(\mathbb{D}^{2n-2}, id)}\})$  in
the $\varepsilon$--partial open book associated to $(\mathbb{S}^{2n-1},Ker \{\alpha_{stot}  \}).$
Moreover, the mapping tours
$\mathcal{M}T(\mathbb{D}^{2n-2}, id)$ is properly embedded close to the convex boundary times $\S^1$  in the 
standard region of the $\varepsilon$--partial open book associated to $(\mathbb{S}^{2n-1}, Ker \{\alpha_{stot} \}).$ 
Also notice that  the pull-back of the form $\varepsilon d \theta  + e^s \alpha$ induces the contact structure 
$Ker \{\alpha_{std} \}$ restricted to $\mathcal{M}T(\mathbb{D}^{2n-2}, id).$

Hence, by the Proposition~\ref{pro:connect_sum_embedding}, there exists an iso-contact abstract open book
embedding of \\ $(\mathcal{A}ob (\Sigma \#_b \mathbb{D}^{2n-2}, \phi \# id), Ker \{ \alpha \# \alpha_{std} \}) $ in 
the $\varepsilon$--partial open book associated to  $(M \# \mathbb{S}^{2n-1}, \xi \# \xi_{stot}).$

Since the contact abstract open book 
$\mathcal{A}ob(\Sigma \#_b \mathbb{D}^{2n-2}, Ker \{\alpha_{(\Sigma \#_b \mathbb{D}^{2n-2}, \phi \# id)} \})$
is contactomorphic to $(M, \xi)$ and $(M \# \mathbb{S}^{2n-1}, \xi \# \xi_{stot})$ -- by definition --
is $(M, \xi^{stot}),$ the proposition follows.
\end{proof}

We now proceed to establish the 
Lemma~\ref{lem:M_in_M_times_D^2} and the Lemma~\ref{lem:embd_std_sphere_in_nbhd_stot_sphere}.

\subsection{Proof of the Lemma~\ref{lem:M_in_M_times_D^2}}\mbox{}

The purpose of this sub-section is to establish the Lemma~\ref{lem:M_in_M_times_D^2}.

\begin{figure}[ht]
\begin{center}
\psfrag{A}{$\hat{f}_{0}(\Sigma)$}
\psfrag{B}{$f_{0}(\Sigma)$}
\psfrag{C}{$B$}
\psfrag{D}{$ [a,b] \times \mathcal{N}(B) $}
\psfrag{E}[r]{$\{b\} \times \mathcal{N}(B)$}
\psfrag{F}{$\{a\} \times \mathcal{N}(B)$}
\includegraphics[width=5cm,height=4cm]{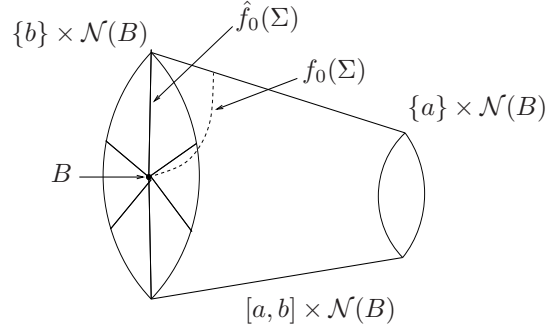}
\caption{The figure depicts a neighborhood $\mathcal{N}(B)$ of  the binding of an open book. We depict 
the vertical line connecting the binding $B$ with the boundary of the circle as a part of the collar 
of $\hat{f}_0(\Sigma)$ while  the dashed curve depicted $f_0(\Sigma)$ obtained after 
pushing $\hat{f}_0(\Sigma)$ in the interior.}
\label{fig:symplectization_binding}
\end{center}
\end{figure}

\begin{lemma}\label{lem:construction_of_isotopy}
Let $(N^{2n-1}, Ker \{\alpha  \}) = (\mathcal{A}ob(\Sigma, \phi), Ker\{\alpha_{(\Sigma,\phi)}\}) $ be a contact abstract open book.  Let $\varepsilon > 0$ be such that $[0,\varepsilon ] \times \partial \Sigma$
is the collar of $\partial \Sigma$ which satisfies the following:
\begin{enumerate}

\item The symplectomorphism $\phi$ is the identity when restricted to this collar. 

\item The form $\alpha_{\Sigma, \phi}$ restricted to this collar is the form 
$e^s \lambda$ for a contact form $\lambda$ defined on 
$\{0\} \times \partial \Sigma.$

\end{enumerate}

 Then,  there exists a  family $f_{(c,t)}$ 
of embeddings of $\Sigma$  in the symplectic manifold $\left([a,b]\times N, d ( e^s \alpha) \right),$  which satisfies the following properties:

 \begin{enumerate}
  \item The family $f_{(c,t)}$ is smooth in both $c$ and $t.$
  
  \item $f_{(c,t)}( \Sigma)$ is a properly embedded symplectic sub-manifold of 
  $\left((a,b] \times N, d (\mathrm{e}^s \alpha)\right)$ for  every $c$ and
   $t.$ 
  
  \item $f_{(c,1)}(\Sigma) = f_{(c,0)}(\Sigma)$ and  
  $f_{(c,1)}^{-1} \circ f_{(c,0)} = \phi.$
  
  \item $\partial f_{(c,t_1)}(x) = \partial f_{(c,t_2)}(x)$ for all $x$ and for any pair of reals
  $t_1,t_2 \in [0,1]$ in the 
  collar neighborhood   $[0, \epsilon] \times \partial \Sigma.$
  
  \item The embedding is such that the complement of the collar
  $f_{(c,t)}\left(\Sigma \setminus (0, \varepsilon] \times \partial \Sigma \right)$ is contained in $ \{c \} \times N.$

  \item The form $f_{(c,t)}^*\left( d (\mathrm{e}^s \alpha)\right)$ restricted to the collar is a part of the symplectization of
  the contact  manifold $(\{0\} \times \partial \Sigma, Ker \{ \mathrm{e}^c \lambda |_{\{0\} \times \partial \Sigma} \}).$  Furthermore,  the primitive of the
  symplectic form  when restricted to the convex boundary 
  $ \{\varepsilon\} \times \partial \Sigma$  is the $1$--form 
  $\mathrm{e}^{b+\varepsilon} \lambda$ and   in a neighborhood of $\{0\} \times \partial \Sigma,$  the primitive is given by   $ \mathrm{e}^s (\mathrm{e}^c \lambda).$

\end{enumerate}
 
\end{lemma}

Before we discuss the formal proof, let us discuss briefly the idea behind the proof.
Given a contact abstract open book $(M, Ker \{ \alpha\})$ with page $\Sigma,$ we know that   $\Sigma$ admits an embedding as a page in $M$ at level $\theta$ for any
$\theta \in S^1.$ Call the image of the embedding as $\Sigma_{\theta}.$ This embedding is symplectic and its boundary is the binding $B$. Hence, for any $c \in [a,b]$ there exists a piece-wise linear embedding of $\Sigma$ in $[a,b] \times M$ consisting of 
$\Sigma_{\theta} \cup [c,b] \times B,$ where we regard $\Sigma_{\theta}$ as 
embedded in $\{c\} \times M.$ Our main observation is that we can smoothen the corner
along the binding $B$ to produce a symplectic embedding in $([a,b] \times M, d(\mathrm{e}^s \alpha)).$

\begin{proof}
 For any $c \in (a, b],$ since $\mathrm{e}^c \alpha$ is supported by the open book decomposition $\mathcal{A}ob(\Sigma, \phi).$ 
This implies that  there exists a family $\widehat{f}_{(c,t)}$ of embedding of $\Sigma$ in $\{c\} \times N$ 
which satisfies the following properties:
 
\begin{enumerate}
\item $\widehat{f}_{(c,t)}^* (\mathrm{e}^c d \alpha)$ is a symplectic form on 
$ \Sigma.$
\item $\widehat{f}_{(c,1)} (\Sigma) = \hat{f}_{(c,0)} (\Sigma)$ and $\widehat{f}_{(c,t_1)}(x) = \widehat{f}_{(c,t_2)}(x)$ 
for every 
$x$ in a collar of $\partial \Sigma$ and for any $t_1,t_2 \in [0,1].$
\item  $\widehat{f}_{(c,0)} \circ \widehat{f}_{(c,1)}^{-1}$ is the symplectomorphism $\phi$ of 
$(\Sigma, \widehat{f}_{(c,0)}^* d ( \mathrm{e}^c \lambda))$ 
 \end{enumerate}

Fix a $t_0 \in [0,1].$
Let us denote by  $\Sigma_{t_0}$ the image $\widehat{f}_{(c,t_0)}(\Sigma).$ Notice that $\Sigma_{t_0}$ is a 
symplectic sub-manifold of $\left([a,b] \times N, d (\mathrm{e}^s \alpha)\right)$ which is contained in $\{c\} \times N.$ Observe that by the definition of 
the abstract open book, we get that the form $e^s \alpha$ restricted to
the collar
$\widehat{f}_{(c,t_0)}([0, \varepsilon] \times \partial \Sigma)$  of $\Sigma_{t_0}$ 
is  pulled back via the symplectomorphism $\widehat{f}_{(c,t_0)}$  
to the $1$--form $\mathrm{e}^r\mathrm{e}^c \lambda,$ where $r \in [0, \varepsilon],$
on the collar $[0, \varepsilon] \times \Sigma.$

We now describe how to  re-embed the collar $[0, \varepsilon] \times \Sigma$ in the symplectic manifold  by an embedding $F$  such that it satisfies the following properties:

\begin{itemize}
 \item  There exists a $\delta \in (0, \varepsilon )$ such that  $F = \widehat{f}_{(c,t_0)}$
 when restricted to $[0, \delta] \times \partial \Sigma .$
 
 \item  The pull-back of the form $ \mathrm{e}^s \alpha$ by $F$ induces the form $e^{b + \varepsilon} \lambda$ on 
 $\{\varepsilon\} \times \partial \Sigma.$
\end{itemize}

In order to achieve this, consider a pair of functions $f: [0, \varepsilon] \rightarrow [a,b]$ and $g: [0, \varepsilon] \rightarrow [0, \varepsilon]$ such that $f$ is constant $c$ near $0$ and increases  to $b$ while $g$ is the identity near $0$ and is the  constant $\varepsilon$ in a small neighborhood of the point $\varepsilon.$ See Figure~\ref{fig:f_and_g}
for graphs of $f$ and $g.$ We now define the embedding $F$ as $F(t, x) = (f(t), g(t), x).$ See Figure~\ref{fig:symplectization_binding}
for a pictorial description of the embedding $F.$

%
%

\begin{figure}[ht]
\begin{center}
\psfrag{f}{$f$}
\psfrag{g}{$g$}
\psfrag{a}[r]{$a$}
\psfrag{b}[r][b]{$b$}
\psfrag{e1}{$\varepsilon $}
\psfrag{O1}[r][b]{$(0,0)$}
\psfrag{O2}[r][b]{$(0,0)$}
\psfrag{e2}[c][b]{$\varepsilon$}
\psfrag{e3}[c][b]{$\varepsilon$}
\includegraphics[width=10cm,height=4cm]{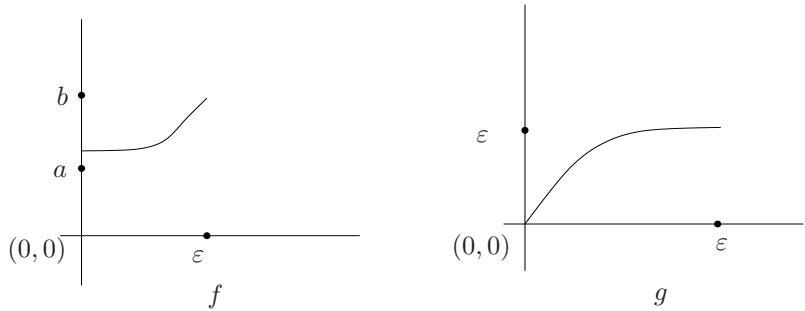}
\caption{The figure depicts graphs of functions $f$ and $g$.}
\label{fig:f_and_g}
\end{center}
\end{figure}

We now define the embedding $f_{(c,t_0)}$ of $\Sigma$ using $F$ and $\widehat{f}_{(c, t_0)}$ as:

$$f_{(c,t_0)}(x) =
\begin{cases}
 F(x), & \text{for} \hspace{0.2cm} x \in [0, \varepsilon] \\
\widehat{f}_{(c, t_0)}, & \text{otherwise.}
\end{cases}
$$

Observe that since $t_0$ is arbitrary, doing the construction parametrically, 
 we get the family $f_{(c,t)}$ from the family $\hat{f}_{(c,t)}$  
with the required properties. This completes our argument.

\end{proof}



\begin{lemma}\label{lem:power_of_given_monodromy_embeds}
 Let $(N, Ker \{\alpha \}) = (\mathcal{A}ob(\Sigma, \phi), Ker\{\alpha_{(\Sigma, \phi)}\})$  be a contact abstract open book.  Let $m = 1$ or $-1.$ Let $(M, Ker\{\alpha_m\})$ be the contact abstract open book $(\mathcal{A}ob(\Sigma, \phi^m), Ker \{\alpha_m = \alpha_{(\Sigma, \phi^m)}\}).$ Let $\mathcal{M}T(\Sigma, \phi^m)$ be the contact mapping torus associated to the contact abstract open 
 book $M.$ Consider the  manifold 
 $(\S^1 \times [a,b] \times N, Ker \{\alpha_K  = K d \theta + \mathrm{e}^s \alpha\}).$ If  $\pi_m: \mathcal{M}T(\Sigma, \phi^m) \rightarrow \S^1$ denotes the bundle projection, then there exists a $K_0 >0$
and a contact embedding of $\mathcal{M}T (\Sigma, \phi^m)$ in 
$(\S^1 \times [a,b] \times N, Ker \{\alpha_K \})$ for all  $K \geq K_0,$ 
which satisfies  the following properties:

\begin{enumerate}
 \item  The following diagram commutes:

$$ \xymatrix{
 \mathcal{M}T(\Sigma, \phi^m)   \ar[rd]^{\pi_m} \ar[r]^F &  
 \S^1 \times [a,b] \times N \ar[d]^{\pi_2}  \\
 & \S^1 .
} $$

 \item For a fixed $c \in (a,b),$ the fiber  $\pi_m ^{-1} ( \theta )$ is a symplectic sub-manifold of the symplectic manifold 
 $([c,b] \times N, d (\mathrm{e}^s \alpha))$ for every $\theta \in \S^1.$

 \item The contact structure induced on the embedded $\mathcal{M}T(\Sigma, \phi^{m})$ by $\alpha_K$ is contactomorphic to
 the contact structure $ Ker \{\alpha_m \}$ restricted to $\mathcal{M}T(\Sigma, \phi^m ) \subset N$ by
 a contactomorphism which is the identity when restricted to $\partial \mathcal{M}T(\Sigma, \phi^{m}).$
 
\item When $m = 1,$ we can choose $K_0$ to be arbitrarily small.

\end{enumerate}

\end{lemma}

\begin{proof}

Our first task is to produce an embedding which satisfies the first two properties stated in the statement
of the lemma. In order to achieve this, let us fix an $m \in \{-1, 0, 1\}.$ Now, consider the family $f_t$
of embeddings of $\Sigma$ in the symplectic manifold $([a,b] \times N,d ( \mathrm{e}^s \alpha) ) $ defined 
as follows:

\begin{equation*}
f_t =   
\begin{cases}

 f_{(c,t)} , &  t \in [0, 1] \hspace{0.2cm} \text{and} \hspace{0.2cm} m=1 \\ 
 f_{(c,1-t)} , &  t \in [0,1] \hspace{0.2cm} \text{and} \hspace{0.2cm} m = -1.
 \end{cases}
\end{equation*}

 The family   $f_{(c,t)},$  is the family constructed in the Lemma~\ref{lem:construction_of_isotopy} earlier. 

Consider the embedding of $I \times \Sigma \subset I \times [a,b] \times N$ given by 
$ F(t, x) = (t, f_t(x)).$  Regard $\S^1 \times [a,b] \times N $ as the quotient of $[0,1] \times [a,b] \times N$ 
where  we identify $\{0\} \times [a,b] \times N$ to $\{1\} \times [a,b] \times N.$ Notice that 
due to the third property associated to the family $f_{(c,t)}$ constructed in the Lemma~\ref{lem:construction_of_isotopy},  
we get that $F$ naturally induces an embedding of 
$\mathcal{M}T(\Sigma, \phi^m)$ in $S^1 \times [a,b] \times N$ which satisfies the first two properties mentioned in
the statement. Let us  denote this induced embedding of $\mathcal{M}T(\Sigma, \phi^m)$ in 
$\S^1 \times [a,b] \times N$ by $F_c.$


Next, observe that $F^* (K d \theta)$ is a $1$-- form on $\mathcal{MT}(\Sigma, \phi)$
 which is transverse to the
fiber of the fibration $\pi_m.$ Since we have already  established that  the fibers of 
$\mathcal{M}T(\Sigma, \phi)$ are symplectic sub-manifolds of $([a,b] \times N, d (\mathrm{e}^s \alpha )),$ 
it follows  that for a sufficiently large $K_0$ the pull-back of  the form 
$ K d \theta + e^s \alpha$ by $F$
induces a contact structure on $\mathcal{MT}(\Sigma, \phi^m)$

We now focus on our claim that $K$ can be chosen to be arbitrary small in case $m=1.$  
Observe  that in the calculation of the volume form  
$F^* (K d\theta +  e^s \alpha) \wedge \left(d F^*( K d \theta + e^s \alpha ) \right)^{n-1}$ on $\mathcal{M}T(\Sigma, \phi),$ we get
 two non-zero terms $ F^*(e^s \alpha) \wedge (F^*(d( e^s \alpha))^{n-1}$ and 
 $ F_m^*(K d \theta) \wedge (F^*( d e^s \alpha))^{n-1}.$ 

Since the second term is always positive for any $K$ positive, 
we get that whenever the first term is non-negative,  
the pull-back form defines a contact form on $\mathcal{MT}(\Sigma, \phi)$ for
an arbitrary $K$ positive. Notice that when $m =1$ the term $ F^*(e^s \alpha) \wedge (F^*(d( e^s \alpha))^{n-1}$
is positive and hence, we get what we required.

So far, we have produced a contact embedding of $M$ in $(\S^1 \times [a,b] \times N, d \theta + \mathrm{e}^s \alpha).$
Hence, in order to establish the lemma, we  need to show that the contact structure induced by the form
$F^*(d \theta + \mathrm{e}^s \alpha)$  on $\mathcal{M}T(\Sigma, \phi^m)$ is contactomorphic to the
contact structure $Ker\{\alpha_m\}$ restricted to $\mathcal{M}T(\Sigma, \phi^m).$

In order to see this,  we observe that the embedding $F_c$ is constructed  using the family $f_{(c,t)}$ of embeddings of
$\Sigma$ in the symplectic manifold $([a,b] \times N, d (\mathrm{e}^s \alpha))$ as in 
the Lemma~\ref{lem:construction_of_isotopy} is smooth in the parameter $c.$
This follows from  the first property listed in the statement of the
Lemma~\ref{lem:construction_of_isotopy}. 
Observe that by varying $c \in (a,b]$ and pulling back the 
form $K d \theta + \mathrm{e}^s \alpha$ via the family $F_c$ of embeddings,  we can produce a $1$--parameter
family $\alpha_t, t \in (a,b]$ of contact forms on $\mathcal{M}T(\Sigma, \phi_m)$ such that 
the form $\alpha_b$ is the contact form $e^{b} \alpha$ restricted to $\mathcal{M}T(\Sigma, \phi^m).$
It follows from the Gray's stability theorem~\cite[Theorem:2.2.2]{Ge} that the induced contact structure 
via the embedding $F_c$ is contactomorphic to $ker\{ \alpha \}$ restricted to $\mathcal{M}T(\Sigma, \phi^m).$

This establishes the lemma.
\end{proof}



We are now in a position to establish the Lemma~\ref{lem:M_in_M_times_D^2}.

 \begin{proof}[Proof of the Lemma~\ref{lem:M_in_M_times_D^2}]\mbox{}
  
  Notice that it follows from the Lemma~\ref{lem:power_of_given_monodromy_embeds} applied in 
  the case when $m = 1 $  that  for a
  given $\varepsilon,$ an arbitrary pair of reals $0 <a < b,$ 
   there is a proper iso-contact
  open book embedding of $\mathcal{M}T(\Sigma, \phi)$ inside the contact
  manifold $(\S^1 \times [a,b] \times M, \varepsilon  d \theta + e^s \alpha).$

  Observe  this clearly implies that the $\varepsilon$--partial open book admits an iso-contact
  abstract open book embedding $F$ of $(M, \alpha)$ as desired.
  
  \end{proof}

\subsection{Proof of the Lemma~\ref{lem:embd_std_sphere_in_nbhd_stot_sphere}}\mbox{}

Let us now establish the Lemma~\ref{lem:embd_std_sphere_in_nbhd_stot_sphere}.

\begin{proof}[Proof of the Lemma~\ref{lem:embd_std_sphere_in_nbhd_stot_sphere}]\mbox{}

To begin with  observe that it follows from the Lemma~\ref{lem:M_in_M_times_D^2} that
there exist a contact abstract open book embedding of the standard
contact $(2n-1)$--sphere
$(\mathbb{S}^{2n-1}, Ker\{\alpha_{std}\}) = (\mathcal{A}ob(\mathbb{D}^{2n-2}, id),
\alpha_{(\mathbb{D}^{2n-2}, id)})$ in any $\varepsilon$--partial open book 
$\S^{2n-1} \times \mathbb{D}_{\varepsilon}$ associated to $(\S^{2n-1}, \alpha_{std}).$
Call this embedding $F.$

Next, produce a generalized contact abstract open book -- say $N$ --by performing the band connected
sums of the pages of the $\varepsilon$--partial open book associated to 
$(\S^{2n-1},Ker\{\alpha_{std}\})$ and the $\varepsilon$--partial book associated to
$(\S^{2n-1},Ker\{\alpha_{stot}\})$ which satisfies the following:

We  choose the neighborhood
of the page $P$ of the $\varepsilon$--partial open book associated to
$(\S^{2n-1}, Ker \{\alpha_{std}\})$ to  
perform the band connected sum to be disjoint from the
image  $F(\S^{2n-1}) \cap P.$ 

It now follows from  the Proposition~\ref{pro:connect_sum_embedding} that there exists
an abstract open book embedding of $(\S^{2n-1}, \alpha_{std})$ in the
$\varepsilon$--partial open book associated to 
$(\S^{2n-1} \# \S^{2n-1}, Ker\{\alpha \# \alpha_{stot}\}).$ This clearly
implies the lemma.


\end{proof}

The only proposition left to establish that is used in the proof of the Theorem~\ref{thm:h-principle} is the
Proposition~\ref{pro:ot_sphere_in_stand_sphere}. In the next sub-section we establish this.

\subsection{Proof of the Proposition~\ref{pro:ot_sphere_in_stand_sphere}}\mbox{}
 
Before we establish this proposition, we would like to point out some
historical developments  related to embedding of overtwisted contact spheres in standard contact spheres. 

 K. Niederkr\"uger and  F. Presas pointed out to  us that the Proposition~\ref{pro:ot_sphere_in_stand_sphere} follows by a straight forward generalization
 of  Example 1.(b) given in \cite{NP}. This fact was known to them when they
wrote the article \cite{NP}.

Embedding of the standard overtwisted sphere in the standard $5$--sphere was
also established in \cite{Mr1} by A. Mori.  This was then generalized by J. Etnyre
and R. Fukuwara in \cite{EF}, where they showed that any overtwisted sphere
embeds in the standard contact $S^5.$

\begin{proof}[Proof of the Proposition~\ref{pro:ot_sphere_in_stand_sphere}]\mbox{}

Consider the standard  contact sphere $(\mathbb{S}^{2n-1}, \xi_{std})$ as the contact abstract open book
$\mathcal{A}ob(T^*\mathbb{S}^{n-1}, \tau).$ Observe that for any large $K,$ there is an iso-contact embedding
of the $K$--partial open book associated to 
$(\mathbb{S}^{2n-1}, Ker \{\alpha_{std} \})$ in 
the standard contact sphere $\S^{2n+1}.$ This is 
because the standard $\mathbb{S}^{2n-1}$ appears as the binding of the trivial open book supporting the
standard contact form on $\S^{2n+1}.$ 

Since $T^*\mathbb{S}^{n-1}$ is a page of the standard open book of $\S^{2n-1}$, 
it follows from the Lemma~\ref{lem:power_of_given_monodromy_embeds}
applied in the case when $\phi$ is $\tau$ and $m = -1,$  that there is an iso-contact abstract
open book embedding of $(\mathbb{S}^{2n-1}, \xi_{stot})$ in the $K$--partial open book associated to 
$(\mathbb{S}^{2n-1}, Ker \{\alpha_{std}\})$ for a very large $K$. Hence the Proposition.

\end{proof}

In the next couple of sections, we   will give applications of the 
Theorem~\ref{thm:h-principle}.

\section{Contact embedding of $3$-manifolds in the standard contact $\S^5$}

The purpose of this section is to show that every contact $3$--manifold $(M, \xi)$ contact embeds in
$(\S^5, \xi_{std}),$ provided the first Chern class of $\xi$ is zero and $M$ has
no $2$-torsion in $H^2(M, \mathbb{Z}).$ The result essentially follows from 
\cite[Theorem:1.20]{EF} and the Proposition~\ref{pro:h-principle_weak}. However, for the
sake of completion, we provide a slightly more detailed argument. We begin this section by
reviewing a few fact about homotopy classes of plane fields on an orientable $3$--manifold.

\subsection{Homotopy classes of oriented plane fields on orientable $3$-manifolds}\label{sec:d_3}
\mbox{}

Let $\xi$ be an oriented
$2$--plane field on a closed oriented $3$--manifold $M.$
Recall that  any two such plane fields are homotopic over the $1$--skeleton of a triangulation  of
$M.$ Gompf in \cite{Go} established that when $M$ has no two torsion in $H^2(M, \mathbb{Z}),$ 
the first Chern class $c_1(\xi)$ completely determines homotopy of plane fields over the  $2$--skeleton.
See \cite[Theorem:4.5]{Go}. 

It also follows from  \cite[Theorem:4.5]{Go} that if $c_1(\xi) = 0,$ then homotopy over 
the $3$--skeleton is completely determined by the $3$--dimensional invariant 
$d_3(\xi),$ which is defined as follows:

It was shown in \cite{Go} that it is possible to choose an almost complex manifold $(X,J)$ whose
complex tangencies are $(M, \xi).$  Given this one defines $d_3(\xi)$ as:

$$ d_3 (\xi) = \frac{1}{4}\left( C_1^2(X,J) - 3 \sigma (X) - 2(\chi (X) -1)\right).$$

We would like to point out that this formula is slightly different from the one given in \cite{Go}, as
we are subtracting $1$ from the Euler characteristic of $X$ in the formula. This is just to ensure that the formula for $d_3$ is additive when one considers the connected sums. More precisely,

Let $(M_1, \xi_1)$ be a contact manifold with $c_1(\xi_1) = 0$ and $(M_2, \xi_2)$ be another contact
manifold with $c_1(\xi_2) = 0,$ then for the contact connected sum $(M_1 \# M_2, \xi_1 \# \xi_2),$ 
we have

$$ d_3(\xi_1 \# \xi_2) = d_3(\xi_1) + d_3(\xi_2) . $$

To begin with, we need the following result of Etnyre and Fukuwara established in \cite{EF}. For the sake of
completeness, we will provide a short sketch of the proof of this result here.

\begin{theorem}[Etnyre and Fukuwara]\label{thm:Etnyre_Fukuwara}
\mbox{}

\noindent Let $M$ be a closed $3$--manifold. If $M$ is orientable,  
then there exists an embedding of $M$ in $\S^5$ such that  the contact structure  $\xi_{std}$
on $\S^5$ induces a contact structure on $M.$
\end{theorem}

\begin{proof}
 To begin with, we  observe  that if there exists an embedding $F:M \rightarrow \S^3 \times D^2$ given 
 by $F(x) = (f_1(x), f_2(x))$ which satisfies the
 following properties:
 
\begin{enumerate}
 \item The map $f_1: M \rightarrow \S^3$ is a branch covering,
 \item The branch locus $L$ in $\S^3$  for the branch cover $f_1:M \rightarrow \S^3$  is transversal to the 
 standard contact structure on $\S^3.$
\end{enumerate}
 
Then,  there exists an $\varepsilon_0 >0$ such that for all $\varepsilon$ less than $\varepsilon_0,$ the embedding
$F_{\varepsilon}:M^3 \rightarrow \S^3 \times D^2$ given by $F_{\varepsilon}(x) = (f_1(x), \varepsilon f_2(x))$ is 
a contact embedding of $M$ in $(\S^3 \times D^2, \{\alpha_{std} + r^2 d \theta\} = 0).$ See \cite{EF} for the
computation establishing that the  pulled back form $F_{\varepsilon}^*(\alpha + r^2 d \theta)$, in fact, induces
a contact structure on $M.$

Now, the Remark $3$ on the page 375 of \cite{HLM}  and the fact that transversality is a
generic property implies that there exists an embedding of $M$ in $\S^3 \times D^2$ satisfying the 
two properties mentioned above. This clearly implies that in an arbitrarily 
small neighborhood of contact $(\S^3, \xi_{std})$ inside $(\S^5, \xi_{std})$  admits an embedding of $M$ such that 
it is a contact embedding.
 
\end{proof}

We are now in a position to  prove   the Theorem~\ref{thm:cont_3-fold_in_S^5}. 
Recall that the Theorem~\ref{thm:cont_3-fold_in_S^5} states that  a necessary and sufficient condition for 
an iso-contact embedding of a contact $3$--manifold $(M, \xi)$ in $(\S^5, \xi_{std})$ is that
$c_1(\xi) = 0$ provided $M$ has no
$2$--torsion in $H^2(M,\mathbb{Z}).$ In case, $M$ has a $2$-torsion in $H^2(M, \mathbb{Z})$, 
the statement claims that 
there is a homotopy class $[\xi]$ of  plane fields  such that $M$ with every contact structure homotopic to
a plane field in the class $[\xi]$ over the $2$--skeleton of $M$ admits an iso-contact embedding in $(\S^5, \xi_{std}).$


\begin{proof}[ Proof of the Theorem~\ref{thm:cont_3-fold_in_S^5}] \mbox{}

We know from \cite{Ka} that $c_1(\xi) = 0$ is a necessary condition for 
having an iso-contact embedding of any contact $(M, \xi)$ in $(\S^5, \xi_{std}).$
We  know from  the Theorem~\ref{thm:Etnyre_Fukuwara} that there exist a contact structure $\eta$ on 
every $3$--manifold $M$ with $c_1(\eta) = 0$ such that $(M, \eta)$ admits an iso-contact embedding in 
$(\S^5, \xi_{std}).$

In case, $M$ has no $2$--torsion in $H^2(M, \mathbb{Z}),$ it follows from the \cite[Theorem:4.5]{Go} 
that every overtwisted  contact structure $\eta_2^{ot}$ on $M$ which is homotopic to $\eta$ over a
$2$--skeleton of $M$ can be obtained by making a contact connected sum of $M$ with a suitably chosen 
overtwisted $\S^3.$ We already know from \cite[Theorem:1:20]{EF} that every contact
$\S^3$ embeds in $(\S^5, \xi_{std}).$ Hence, we conclude that if $(M, \eta)$ iso-contact embeds in 
$(\S^5, \xi_{std}),$ then so does $(M, \eta_2^{ot})$ provided $\eta^{ot}$ is an overtwisted contact
structure on $M$ which is homotopic to $\eta$ over the $2$--skeleton of $M.$ But, this implies that
every $(M, \eta^{ot})$ iso-contact embeds in $(\S^5, \xi_{std}),$ provided the first
Chern class of $\eta^{ot}$ is zero.
The case of no $2$--torsion in $H^2(M,\mathbb{Z})$ is now a straightforward consequence of the 
Corollary~\ref{thm:h-principle}.

In case, $M$ has a $2$--torsion in $H^2(M, \mathbb{Z})$ -- by an argument
similar to the one discussed above  -- it is clear that every overtwisted contact structure 
$\xi^{ot}$ on $M$ such that $\xi^{ot}$ is homotopic to $\eta$ as an almost contact plane field over a 
$2$--skeleton on $M$ admits
iso-contact embedding in $(\S^5, \xi_{std}).$  Again, applying the Theorem~\ref{thm:h-principle}, we 
conclude that every contact structure homotopic as a plane field over $2$--skeleton to $\eta$ admits
an iso-contact embedding in $(\S^5, \xi_{std}).$ This completes our argument.

\end{proof}

 \section{Embeddings of simply connected $5$--manifolds in $(\S^7, \xi_{std})$}\label{sec:5-manifold}
 
We begin this section by observing the following:

\begin{proposition}\label{pro:sphere_iso_cont_embedding}
 Let $\xi$ be a contact structure on $\mathbb{S}^{2n-1}.$ If $\xi$ is co-orientable and
 homotopic as an almost-contact structure to the standard contact structure on $\mathbb{S}^{2n-1}$, then
 $(\mathbb{S}^{2n-1}, \xi)$ admits
 an iso-contact embedding in $(\mathbb{S}^{2n+1}, \xi_{std}).$ In particular, every contact
 $(\S^5, \xi)$ iso-contact embeds in $(\S^7, \xi_{std}).$
\end{proposition}

\begin{proof}
 The first part  of the proposition is an immediate 
 consequence of the Proposition~\ref{cor:h-principle_for_embedding_in_spheres}.
 In order to establish the second part, recall that there exists a unique almost-contact class
 on $\S^5.$ This was established in \cite{Ge1}. See also \cite{Ha}. But this implies $\xi_{stot}$ is homotopic as 
 an almost-contact
 plane field to $\xi.$ Hence the theorem.  
\end{proof}

Next, we show that any contact structure on $\S^2 \times \S^3$ with trivial first Chern class  iso-contact
embeds in $(\S^7, \xi_{std}).$ More precisely, we establish:

\begin{lemma}\label{lem:embedding_S^2_times_S^3}
 Let $\xi$ be a co-orientable contact structure on $\S^2 \times \S^3.$ The contact manifold
 $(\S^2 \times \S^3, \xi)$ iso-contact embeds in $(\S^7, \xi_{std})$ if and only if the first
 Chern class $c_1(\xi)$ of the contact structure is zero. 
\end{lemma}

\begin{proof}
 Recall that in \cite{Ge1, Ha} it is  established that two almost-contact plane fields $\xi_1$ and $\xi_2$ are
 homotopic as almost-contact structures if and only if their first Chern classes coincide. 
 
 Next, Kasuya in \cite{Ka} showed that a necessary condition
 for a contact manifold $(M^{2n+1}, \xi)$ to admit an iso-contact embedding in $(\S^{2n+3}, \xi_{std})$ is
 that $c_1(\xi) = 0.$
 
 Hence, from the Corollary~\ref{cor:h-principle_for_embedding_in_spheres} we can see that  if there exist a contact embedding
 of $\S^2 \times \S^3$ in $(\S^7, \xi_{std})$, then the lemma follows. So, we now show that there
 is a contact  embedding of $\S^2 \times \S^3$ in $(\S^7, \xi_{std})$. 
 
 Notice that the contact abstract open book $\mathcal{A}ob(T^*\S^2, id)$ is contact manifold
 diffeomorphic to $\S^2 \times \S^3.$ Clearly, $\mathcal{A}ob(T^*\S^2, id)$ iso-contact open book
 embeds in the contact abstract open book $\mathcal{A}ob(D^6, id)$ as there is
 a  symplectic embedding of $\mathcal{D}T^*\S^2$ in $\S^5.$ Since contact 
 abstract open book $\mathcal{A}ob(D^6,id)$ is contactomorphic to $(\S^7, \xi_{std}),$ the 
 theorem follows.
 \end{proof}

 It was established by H. Geiges in \cite[Chapter--8]{Ge} that a necessary condition to produce a contact 
structure on any five manifold is that the third integral Steifel-Whitney class $W_3$ is zero.
D. Barden in \cite{Ba} had given a complete classification of simply connected $5$--manifolds.
Using this classification,  it is easy to list all the simply connected prime $5$-manifolds with vanishing 
$W_3.$ We now proceed to describe this list. First of all, recall that for each $2 \leq k < \infty,$ there exists a 
unique prime simply connected manifold $M_k$ characterized by the property that 
$H_2(M_k, \mathbb{Z}) = \mathbb{Z}_k \oplus \mathbb{Z}_k.$ Next, recall that there exists a unique non-trivial 
orientable real rank $4$ vector-bundle over $\S^2.$ By  $\S^2 \widetilde{\times} \S^3,$ we denote the
unit sphere bundle associated to this vector-bundle.

We are now in a position to state Barden's theorem that we will need to establish the 
Theorem~\ref{thm:simply_connected_5-folds}.

\begin{theorem}[Barden]\label{thm:barden}
Every closed simply connected almost contact $5$--manifold can be uniquely decomposed into a connected sum of 
prime manifolds $M_k, 2 \leq k < \infty$, $\S^2 \times \S^3$ and  $\S^2 \widetilde{\times} \S^3.$ 
Furthermore, the decomposition has no  copy of $\S^2 \widetilde{\times} \S^3$ provided
the second Steifel-Whitney class is zero.
\end{theorem}

\begin{proof}[ Proof of the Theorem~\ref{thm:simply_connected_5-folds}]\mbox{}

Notice that it is sufficient to establish that any $(M, \xi)$ satisfying 
the hypothesis with $c_1(\xi) = 0$ admits an iso-contact embedding in 
$(\S^7,\xi_{std}).$

Let $M$ be a closed simply connected $5$--manifold  with $w_2(M) = 0$ and let $\xi$ be a contact structure on it with $c_1(\xi) = 0.$ In order to establish the
Theorem~\ref{thm:simply_connected_5-folds}, we first show that $M$ admits a contact embedding in $(\S^7, \xi_{std})$ such that the induced contact structure has its first Chern class $0.$

Notice that if $M$ is as in the hypothesis, then it follows from Theorem~\ref{thm:barden} of Barden stated above that in its connected sum 
decomposition there is no $\S^2 \widetilde{\times} \S^3$ factor. See also, \cite[Theorem:8.2.9]{Ge} for a proof of this.  
Next, we have already observed that if $M = N_1 \# N_2 \# \cdots \# N_l$ and   each $N_i$ contact embeds
in $(\S^7,\xi_{std})$, then there exist a contact embedding of $M$ in $(\S^7, \xi_{std}).$ 

We have already shown in the Lemma~\ref{lem:embedding_S^2_times_S^3} that $\S^2 \times \S^3$ contact embeds in $(\S^7, \xi_{std}).$ Hence, in order to show that $M$ contact embeds in $(\S^7, \xi_{std}),$
we just need to show that each prime manifold $M_k$ described in the Theorem~\ref{thm:barden} above
must contact embed in $(\S^7, \xi_{std}).$ It is well known  that each $M_k$ is a Briskorn $5$--sphere. 
Hence, they admit  contact embedding in $(\S^7, \xi_{std}).$ See, for example, \cite[Rmk:4.2]{Ko}.

Thus, we have shown that every simply connected $5$--manifold satisfying the
hypothesis admits a contact embedding in $(\S^7, \xi_{std}).$  Next,
recall that if a $5$--manifold admits a formal contact embedding in $(\S^7,\xi_{std}),$ then it was shown in \cite{Ka} that the first Chern class of the induced contact structure has to be trivial.

Finally, observe that it was established in \cite[Chpt--8]{Ge,Ge1} ( see also
\cite[chpt-VII]{Ha} for a precise formulation) that any two contact structures on a closed  simply connected $5$--manifold having first Chern classes trivial  are homotopic as almost-contact structures. It  now follows from the  Proposition~\ref{cor:h-principle_for_embedding_in_spheres} that
$(M, \xi)$ admits an iso-contact embedding in $(\S^7, \xi_{std}).$

\end{proof}

%
%
%
%
%
%
%
%
%
%
%
%
%

\end{document}